\documentclass{llncs}
\usepackage{amssymb}
\usepackage{amsmath}
\usepackage{graphicx}
\usepackage{color,hyperref}
\definecolor{darkblue}{rgb}{0.0,0.0,0.3}
\hypersetup{colorlinks,breaklinks,
            linkcolor=darkblue,urlcolor=darkblue,
            anchorcolor=darkblue,citecolor=darkblue}

\newcommand{\identity}{\epsilon}

\newtheorem{fact}{Fact}

\DeclareMathOperator{\Perm}{Perm}

\title{On Straight Words and Minimal Permutators in Finite Transformation Semigroups}
\titlerunning{On Straight Words and Minimal Permutators}
\author{Attila Egri-Nagy and Chrystopher L. Nehaniv}
\institute{Royal Society Wolfson BioComputation Research Lab\\ Centre for Computer Science \& Informatics Research\\ University of Hertfordshire\\ Hatfield, Hertfordshire AL10 9AB\\ United Kingdom
\\
\email{\{A.Egri-Nagy,C.L.Nehaniv\}@herts.ac.uk}
\thanks{Partial support for this work by the OPAALS EU project FP6-034824 
is gratefully acknowledged.}}
\begin{document}

\maketitle

\begin{abstract}
Motivated by issues arising in computer science,
we investigate the loop-free paths from the identity transformation and corresponding straight words in the Cayley graph of a finite transformation semigroup with a fixed generator set. 
 Of  special interest are words that permute a given subset of the state set. 
Certain such words, called minimal permutators, are shown to comprise a 
code, and the straight ones comprise a finite code. Thus, words that permute a given subset are uniquely factorizable as products of the subset's minimal permutators, and these can be further reduced to straight minimal permutators. This leads to insight into structure of local pools of reversibility in transformation semigroups in terms of the set of words permuting a given subset. These findings can be exploited in practical calculations for hierarchical decompositions of finite automata.
As an example we consider groups arising in biological systems.
\end{abstract}

\section{Introduction}
  From the computational perspective it is very important to know how a particular element of a transformation semigroup can (efficiently) be generated. Of special interest are elements of the semigroup that permute a subset of the state set, as the hierarchical decomposition of the semigroup \cite{primedecomp68} depends on the group components \cite{ciaa2004poster,ActaCybernetica2005}. Here we study the ways in which a particular transformation can be expressed without any redundancy. These generator words head towards the target transformation without any repetitions, so they are called \emph{straight}. Straight words also encode the information describing all possible ways that particular semigroup element can be generated. 


\subsection{Notation} 
For a finite transformation semigroup $(X,S)$ we fix a generator set of transformations $T=\{t_1,\ldots, t_n\}$, so $S= \langle T\rangle$. We also consider the generators as symbols, thus a finite product of the generator elements becomes a word in $T^+$ (the free semigroup on generators $T$ whose associative binary operation is  concatenation). It is then convenient to consider the empty word $\identity$ as the identity map. We need to distinguish between the word (often thought of as a sequence of input symbols) and the transformation it realizes: for the word we just write the generator symbols in sequence ${t_{i_1}\ldots t_{i_n}}\in T^+$ while the transformation is denoted by $\overrightarrow{t_{i_1}\ldots t_{i_k}}\in S$, where the arrow indicates the order in which the generator elements are multiplied and emphasizes that is a mapping.
 
For transformations, we either use the usual 2-line notation for mappings, or if it would become too space consuming we apply the linear notation suggested in \cite{ClassicalTransSemigroups2009}. This is a natural extension of the cyclic notation of permutations. Considering the mappings as digraphs, each transformation consists of one or more components. Each component contains a cycle (possibly a trivial cycle). Unlike the permutation case, the points in the cycle can have incoming edges, denoted by $$[ \text{ source}_1,\ldots,\text{source}_m\ ;\text{ target }]$$ where target is the point in the cycle. If a source point also has incoming edges from other points the same square bracket structure is applied again recursively. We can say that the points in the cycle are sinks of trees. Therefore the brackets indicate the existence of a nontrivial permutation of the sink elements of the trees, but not of their sources:
$$ \big([ \text{ sources}_1;\text{ target}_1\ ],\ldots,[ \text{ sources}_k;\text{ target}_k\ ] \big).$$   
This corresponds to the cycle $(\text{ target}_1,\ldots,\text{ target}_k)$ but at the same time it contains information on transient states. The order is arbitrary if there are more than one component.
(See below for examples.)\footnote{Our notation is slightly different from \cite{ClassicalTransSemigroups2009} as we do not use square brackets for a singleton source.   }

\section{Straight Words}

If the goal is to generate a transformation $s\in S$ as quickly as possible without any digression, then in each step of the generation a new transformation should appear. Also, if a prefix generates the identity map, so strictly speaking we did nothing so far, then the prefix can be discarded. More precisely,

\begin{definition}[Straight Words]\label{def:straightword}
Let $s\in S$ be a transformation generated by  the word $t_{i_1}\ldots t_{i_m}\in T^+$, so $s=\overrightarrow{t_{i_1}\ldots t_{i_m}}$, then this word is \emph{straight} if 
\begin{equation}\label{eq:identityloop}
\overrightarrow{t_{i_1}\ldots t_{i_k}}\neq \identity,\ k\in \{1,\ldots,m-1\}
\end{equation}
and
\begin{equation}\label{eq:repetitionloop}
\overrightarrow{t_{i_1}\ldots t_{i_k}}=\overrightarrow{t_{i_1} \ldots t_{i_l}}\Rightarrow k=l\ \ \  (1\leq k,l\leq m).
\end{equation}
\end{definition}

\begin{example}[Cyclic (monogenic) semigroup] Let $X=\{1,2,3,4\}$ and $t=(\begin{smallmatrix}1&2&3&4\\2&4&1&2\end{smallmatrix})$, or, in the alternative notation, $t=([[3;1];2],4)$. The semigroup generated by $t$ is $$\langle t\rangle=\{t=(\begin{smallmatrix}1&2&3&4\\2&4&1&2\end{smallmatrix}),
t^2=(\begin{smallmatrix}1&2&3&4\\4&2&2&4\end{smallmatrix}),
t^3=(\begin{smallmatrix}1&2&3&4\\2&4&4&2\end{smallmatrix})\}.$$ $t$, $t^2$ and $t^3$ are straight words, but these are the only ones. Higher powers, like $t^4=t^2$ already contain repeated transformations. This example shows that being straight is not necessarily connected to the formal notion of containing repeated subwords.
\end{example}
\begin{example}[Cyclic group] Let $g=(1,2,3)$ be a permutation, then $g^3=\identity$ is a straight word producing the identity map. This example justifies condition \ref{eq:identityloop} in Definition \ref{def:straightword}, as we allow the identity transformation at the end of a word, but not inside.
\label{ex:cyclicgroup}
\end{example}

\subsection{Alternative defintion using trajectories}
\begin{definition}[Trajectory]
Let $s_1,\ldots,s_n$ be a sequence of semigroup elements, $s_j\in S$. Then the sequence is a \emph{trajectory} if for all $s_j, 1\leq j < n$ there is a generator $t_i\in T$ such that $s_j\cdot t_i=s_{j+1}$.
\end{definition}
A trajectory is a path in the Cayley graph of the semigroup starting at the trivial transformation. We can associate a trajectory with a word.

\begin{definition}[Trajectory of a word]
Given a word $t_{i_1}\ldots t_{i_m}$ the corresponding trajectory is calculated by taking the products of prefixes:
$$\identity,\overrightarrow{t_{i_1}},\overrightarrow{t_{i_1}t_{i_2}},\ldots,\overrightarrow{t_{i_1}\ldots t_{i_m}}. $$
\end{definition}

With trajectories we can obviously give an alternative definition of straight words.

\begin{definition}[Straight words]
A word is \emph{straight} if all the elements of its trajectory are distinct, except the case of loops when the first and the last element coincide (and equal $\identity$).
\end{definition}

\subsection{Straight words and  transformations}
From finiteness it follows that the straight words cannot be extended beyond some  finite length, since there are finitely many elements of the semigroup and each prefix should realize a distinct semigroup element. An obvious bound on the length of the straight words is $|S|$. This bound is reached in Example \ref{ex:cyclicgroup}. We also observe that all semigroup elements can be realized by a straight word. 
\begin{lemma}
Let  $(X,S)$ be a transformation semigroup with states $X$ and semigroup $S$ generated by $T$. 
Each semigroup element $s\in S$ 
can be realized by a straight word in the letters of $T$.
\end{lemma}
\begin{proof}
Let $s=\overrightarrow{t_{i_1}\ldots t_{i_m}}$. If   $t_{i_1}\ldots t_{i_m}$ is not straight then there is $k\neq l$ such that $\overrightarrow{t_{i_1} \ldots t_{i_k}}=\overrightarrow{t_{i_1}\ldots t_{i_l}}$. Suppose that $k<l$. Then the product $\overrightarrow{t_{i_1}\ldots t_{i_k}t_{i_{l+1}}\ldots t_{i_m}}$ still generates $s$, after we cut out $t_{i_{k+1}}\ldots t_{i_l}$.

Similarly, in case an identity appears at some position (not the final one) in a trajectory then the whole prefix can be ignored up to that point.
 If the reduced word is not straight then we can repeat either processes. Due to finiteness this method will stop, and thus produce a straight word generating $s$.
\qed
\end{proof}

Another way to see that there is at least one straight word for each transformation is to observe that the first occurences (but not the subsequent ones) of transformations in a breadth-first generation of $S$ by $T$  are produced by straight words.

\begin{corollary}
Any minimal length word generating $s\in S$ is a straight word.
\end{corollary}

We have seen that for each semigroup element we can give at least one straight generator word. The following example shows that there can be more straight words for a mapping.
\begin{example}[Constant Maps] Let $t_1=(\begin{smallmatrix}1&2\\1&1\end{smallmatrix})$ and  $t_2=(\begin{smallmatrix}1&2\\2&2\end{smallmatrix})$ be two generators, then $t_1$ and $t_2t_1$ are each straight words for  $\overrightarrow{t_1}$, while $t_2$ and $t_1t_2$ are straight and both realize $\overrightarrow{t_2}$. Constant maps render the transformations before them negligible. 
\end{example}



\subsection{Synonym Straight Words}

Different straight words may represent the same transformation.
For example, if we add a second generator, $r=(\begin{smallmatrix}1&2&3&4\\4&2&2&4\end{smallmatrix})$ to Example 1, then clearly $r$ and $t^2$ are words with this properties. Moreover, two different words may have the same trajectories. 

\subsection{Generalization: Straight Paths}

We can study straight words in a more general settings, we look for straight words $w=t_{i_1}\ldots t_{i_m}$ such that $s\cdot \overrightarrow{w} = r$, where $r,s\in S$. Actually these arise as labels of `straight paths' in the Cayley graph of the semigroup between nodes $s$ and $r$, i.e.\ simple paths that do not cross themselves but go directly from $s$ to $r$. We get the special case of straight words when $s=\identity$.

\subsection{Computational Implementation}

Computational enumeration of  straight words can easily be done with a backtrack algorithm. We implemented the search algorithm in the \texttt{SgpDec} software package \cite{sgpdec} in the \texttt{GAP} computer algebra system \cite{gap4}.   

\section{Minimal Straight Words and Permutations of Subsets}

From now on we focus on straight words that induce permutations on a subset of the state set.
The {\em full permutator semigroup} Perm($Y$) for a subset  $Y \subseteq X$ in $(X,S)$ is
$$Perm(Y) = \{s\in S: Y\cdot s=Y  \}.$$  
Elements of $Perm(Y)$ are called also {\em permutators of} $Y$.
$Perm(Y)$ is closed under products, so by finiteness it restricts to a group of permutations acting on $Y$. The restrictions of elements of $Perm(Y)$ to $Y$  thus comprise a permutation group or `pool of reversibility' or `natural subsystem' within the transformation semigroup $(X,S)$. 
However, while  any $s \in \Perm(Y)$ is also defined on all of $X$ it  is not generally a permutation of $X$. The elements of $Perm(Y)$ may agree on $Y$ but disagree on $X\setminus Y$, so $Perm(Y)$ may not itself be a group nor act faithfully on $Y$. We also call a word a {\it permutator word} if it realizes a permutator transformation.

\begin{example}[Cyclic uniquelly labelled digraph as an automaton]\label{ex:uldgcycle}
The generator set consists of 3 elementary collapsings, $T=\{a=1\mapsto 2,\ b= 2\mapsto 3,\linebreak c = 3\mapsto 1\}$. The generated semigroup has 21 elements and, in the notation introduced above, the straight words of the semigroup elements are:
\begin{center}
\begin{tabular}{cc}
$[3;1]$ & $c$ \\ 
$[2;3]$ & $b$ \\ 
$[1;2]$ & $a$ \\ 
$[[2;3];1]$ & $cb$ \\ 
$[[1;2];3]$ & $ba$ \\ 
$[[3;1];2]$ & $ac$ \\ 
$([1;2],3)$ & $cba$ \\ 
$([3;1],2)$ & $bac$ \\ 
$(1,[2;3])$ & $acb$ \\ 
$(1,[3;2])$ & $cbac$ \\ 
$([2;1],3)$ & $bacb$ \\ 
$([1;3],2)$ & $acba$ \\ 
$[[2;1];3]$ & $cbacb$ \\ 
$[[1;3];2]$ & $bacba$ \\ 
$[[3;2];1]$ & $acbac$ \\ 
$[1;3]$ & $cbacba$ \\ 
$[3;2]$ & $bacbac$ \\ 
$[2;1]$ & $acbacb$ \\ 
\end{tabular}
\end{center} 
\noindent
plus the constant maps that are represented by a lot more straight words:

$[1,3;2]$  $abca$, $aca$, $acbabca$, $acbaca$, $acbacbca$, $acbca$, $babca$, 
  $baca$, $bacbabca$, $bacbaca$, $bacbca$, $bca$, $ca$, $cbabca$, $cbaca$, 
  $cbacbabca$, $cbacbca$, $cbca$
 
$[2,3;1]$  $abc$, $acabc$, $acbabc$, $acbacabc$, $acbacbc$, $acbc$, $babc$, 
  $bacabc$, $bacbabc$, $bacbacabc$, $bacbc$, $bc$, $cabc$, $cbabc$, 
  $cbacabc$, $cbacbabc$, $cbacbc$, $cbc$ 

$[1,2;3]$ \ $ab$, $acab$, $acbab$, $acbacab$, $acbacbcab$, $acbcab$, $bab$, 
  $bacab$, $bacbab$, $bacbacab$, $bacbcab$, $bcab$, $cab$, $cbab$, $cbacab$, 
  $cbacbab$, $cbacbcab$, $cbcab$ 

Now let $Y=\{1,2\}$. Then there are exactly 4 straight words permuting $Y$,  $bac$ and $cbac$ realizing the transposition $(1\ 2)$, and $c$ and  $bacbac$ realizing the identity, thus  $\{\overrightarrow{c},\overrightarrow{bac},\overrightarrow{cbac},\overrightarrow{bacbac}\}\subseteq\Perm(Y)$. Note that two of these words are products of two of the others.  If we extend the search also for words that are not straight we can find more permutators. For example $\overrightarrow{baac}=\overrightarrow{bbac}=([3;1],2)$.
\end{example}

A word $w$ is a  {\em minimal permutator} of $Y$ if $w$ represents an element of Perm($Y$) and
 $w$ is not a product of two or more words permuting $Y$. That is, $w\neq w_1 w_2$ for  
any words $w_1$ and $w_2$ representing elements of Perm($Y$). The set of minimal permutators is not necessarily finite, as we can use idempotents to ``pump in the middle'' like $ba^nc$ in Example \ref{ex:uldgcycle}. 
Therefore we turn our attention to the set of \emph{minimal straight permutators} denoted by $M_S(Y)$. 

\begin{fact} \label{fact}
The set $M_S(Y)$ of minimal straight permutator words for $Y$ is finite.
\end{fact}
\begin{proof} 
The assertion easily follows from the fact that straight words are bounded in length.\qed
\end{proof}
Now  we need to show that we do not lose anything by discarding the words that are not straight, i.e.\ we can still generate the full permutator semigroup. We will use the following obvious fact.
\begin{fact}\label{fact:permute}
If  $w = uv $ permutes $Y$ and $u$ permutes $Y$, then $v$ permutes $Y$.
\end{fact}
\begin{theorem} 
In the free semigroup $T^+$ on the generators of $S$ , the minimal 
 permutators $M(Y)$ of $Y$ generate the subsemigroup of all words realizing  elements of $\Perm(Y)$. That is, 
$$\langle M(Y)\rangle = \text{all words representing elements of } \Perm(Y ).$$ 
Moreover, the minimal 
straight permutators $M_S(Y)$ of $Y$ generate a subsemigroup of words realizing all
elements of $Perm(Y)$. 
\end{theorem}
\begin{proof} Let $p = t_1 \ldots t_k$ represent an element of Perm($Y$). We
show $p$ is a product of minimal permutators by induction on $k$.
Either $p$ is a minimal permutator or there is a least $j$ strictly
less than $k$ so that $t_1 \ldots t_j$ permutes $Y$. Now $t_1 \ldots
t_j$ is a minimal permutator of $Y$ and $p = (t_1 \ldots t_j) (t_{j+1}
\ldots t_k)$ with each of the expressions in parentheses permuting
$Y$.  The length of the second word is strictly less than $k$, so by
induction hypothesis, it too can be written as a product of minimal
permutators of $Y$. This proves that an arbitrary word $p$
representing an element of Perm($Y$) can be factored as a product of
minimal permutators of $Y$. Each minimal permutator factor can
be shortened by removing letters if necessary to a straight word (or the empty word).
The result follows.\qed\end{proof}

\begin{theorem} 
Any word $w$ representing a permutator of $Y$ can be factored uniquely into a product of minimal permutators of $Y$.
\end{theorem}
\begin{proof} By the previous theorem, we can write 
$$ w= w_1 \cdots w_k, $$ where each $w_i$ is a minimal permutator word of $Y$.
Suppose $w$ can also be written as
$$ w = w_1' \ldots w'_{\ell},$$
where again each $w_i'$ represents  is a minimal permutator word for $Y$.
We show $\ell = k$ and $w_j=w_j'$ for all $j$ ($1 \leq j \leq \ell$).
If this were not the case, then let  $i$ be the least  index such that $w_i \neq w_i'$.
Without loss of generality, assume $|w_i| \leq |w_i'|$.  It follows then that
$w_i' = w_i v$ for some nonempty word $v$. By Fact~\ref{fact:permute}, $v$ 
represents a permutator of $Y$. But we have then written $w_i'$ as a product of permutator words, 
this contradicts the choice of $w_i'$ as a  minimal permutator. It follows $w_i=w_i'$ for all $i$, and, since
the two factorizations are of the same word, that $\ell=k$.   \qed\end{proof}

In other words,

\begin{corollary}
The minimal permutator words are a code.
\end{corollary}

\begin{corollary}
The minimal permutator straight words are a finite code.
\end{corollary}
The last corollary shows the usefulness of straight words, when looking for permutators instead of an infinite search space we can restrict the search to a finite set of words. 

\begin{fact}
For minimal permutator word $w$, there is a (in general non-unique) straight minimal permutator word $red(w)$ obtained from $w$ by removing some letters such that $
\overrightarrow{w}= \overrightarrow{red(w)}$.
\end{fact}
\begin{proof} Considering the Cayley graph of the transformation semigroup $(X,S)$ with generators $T$. This has vertices $S^1 = S \cup \{\epsilon\}$, where $\epsilon$ denotes the identity mapping on $X$, and edges $s \stackrel{{t}}{\longrightarrow} s'$, where $s'=s\,\overrightarrow{t}$ with $t \in T$, $s,s' \in S^1$.   
Now, by the alternative definition of straight words,  it is clear the a word is straight if and only if the path it labels starting at $\epsilon$ and has no loop (does not visit any node more than once). Noting that adding or removing loops to the path corresponding to a product does not change 
its endpoint, we conclude that removing contiguous subwords from the word $w$ corresponding to loops, iteratively if necessary,  results in a path with no loops, corresponding to a  straight word $w'$ representing the same transformation as $w$.
\qed
\end{proof}

\begin{theorem}
There is a well-defined homomorphism $\phi: M^+(Y) \twoheadrightarrow M_S(Y)^+$ from  the semigroup of permutator words onto the semigroup generated by minimal straight permutator words. Furthermore, $\phi$ is a retraction, i.e.\ $\phi(w)=w$ for all words $w$ in $M_S(Y)$ (and hence is the identity on $M_S(Y)^+$). 
The permutator words
$w$ and $\phi(w)$ act by the same permutation of $Y$, and moreover by the same mapping on $X$, and $\phi(w)$ is a straight word obtained from $w$ having the same trajectory as $w$ except for the removal of loops.
\end{theorem}
\begin{proof}
To get a well-defined homomorphism from minimal permutator code to the straightword minimal permutator code, one only needs to choose some reduction for each minimal permutator (any reduction at all would work).
 The reason why one gets a homomorphism is due to that fact the minimal permutators are a code, hence free generators of a free semigroup, so we need only say where each generator goes and extend uniquely by freeness.
\qed
\end{proof}

The reduction of a minimal permutator to a straight word need not be unique. This comes from the fact that synonym straight words do exist. 
Thus the homomorphism of the theorem need not be unique.

One natural way to choose the reduction $red(w)$ is the following: 
given a minimal permutator word $w$ that is not straight, find
the first node (along its trajectory) that is later repeated. Start deleting letters after the letter that first takes us into this node. Find the last time this node occurs. Delete all letters from there up to and including the one taking us into the node for the last time. This process removes at least one letter since the word was not straight. Repeat the procedure until the resulting word is straight.  This necessarily terminates with a reduced form for $w$, realizing the same transformation by a straight word obtained from $w$ by some excising some subwords (`removing loops' in the trajectory as described). 


\section{A Biological Example}

It seems that in constructing interesting examples the human mind is somewhat contrained and reverts back to special cases. Therefore studying ``naturally occuring'' transformation semigroups can be useful, so here we investigate a biological example.
We should also mention that in exchange  semigroup and automata theory can also provide useful tools for  other sciences \cite{rhodes_book}.

\subsection{The p53-mdm2 regulatory pathway}

\begin{figure}[h]
\centering
\includegraphics[width=\textwidth]{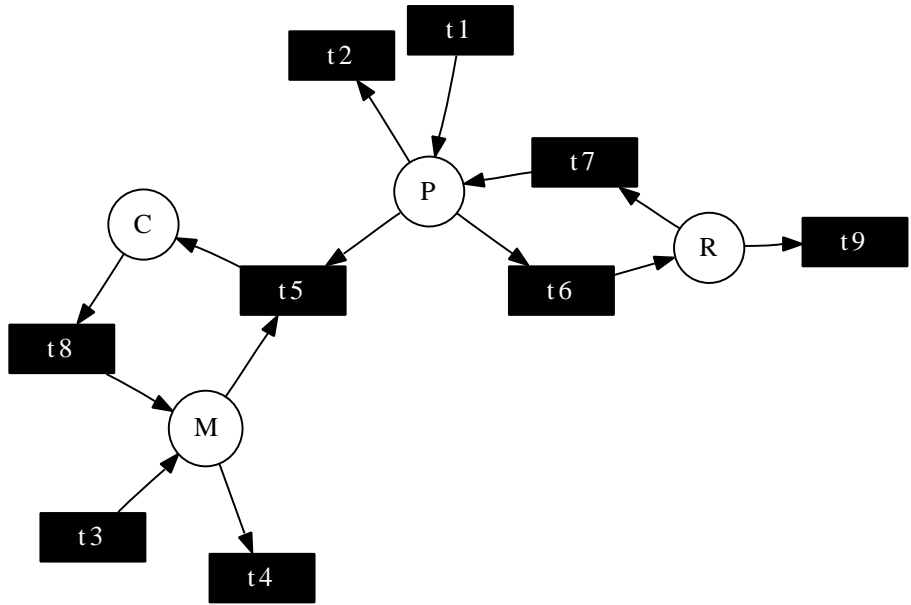}
\caption{\bf \small Petri net for the p53-mdm2 regulatory pathway. P = p53, M = mdm2, C = p53-mdm2, R = p53*.}
\label{fig:parrondo-new-Petrinet}
\end{figure}

\begin{figure}[h]
\centering
\includegraphics[width=\textwidth]{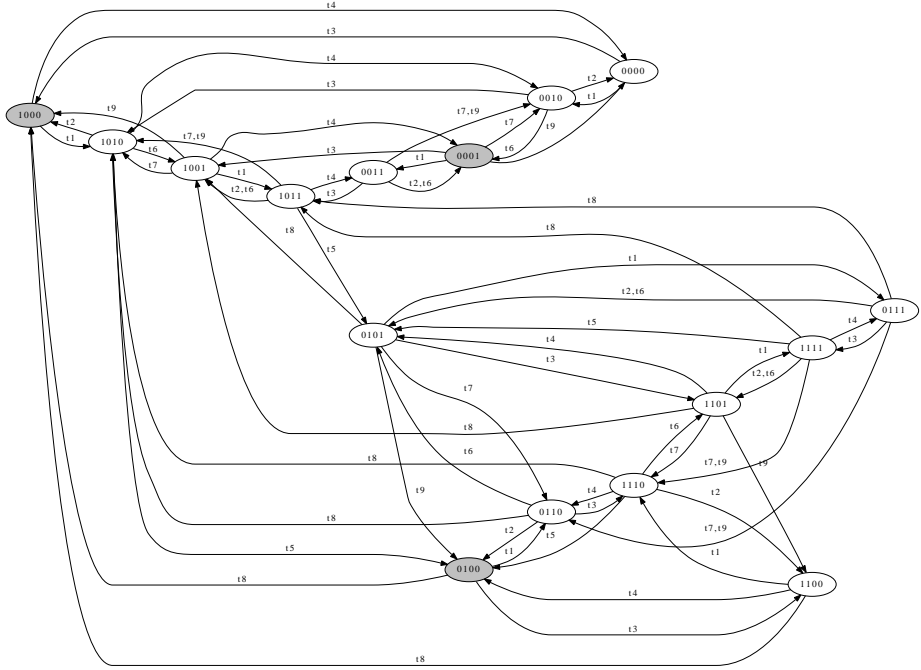}
\caption{\bf \small Automaton derived from 2-level Petri net of the p53 system (16 states). The labels on the nodes encode the possible configurations for M, C, P and R (in this order). 0 denotes the absence (or presence below a threshold), 1 the presence (above the threshold) of the given type of molecule. For instance, 0101 means that C and R are present. The shaded states correspond to the state set $\{3=1000,5=0001,8=0100\}$.}
\label{fig:parrondobinarayfa}
\end{figure}

 Biological networks are frequently modelled by Petri nets and thus it is not difficult to convert such a model to a transformation semigroup \cite{PetriNet2008BioSys}.
Figure~1 shows such a model of the p53-mdm2 regulatory pathway, which is important in the cellular response to ionizing radiation and can trigger self-repair or, in extreme cases, the onset of programmed cell-death (apoptosis). This pathway is involved in ameliorating DNA damage and preventing cancer \cite{p53aref}.
Figure~2 shows the corresponding finite automata with 16 states in which two levels of each of the 4 molecular species involved are distinguished.  Corresponding to the transitions we have the following generator transformations:
\begin{align*}
t_1&=[1;2][3;4][5;6][7;9][8;10][11;12][13;14][15;16]\\
t_2&=[2;1][4;3][6;5][9;7][10;8][12;11][14;13][16;15]\\
t_3&=[1;3][2;4][5;7][6;9][8;11][10;12][13;15][14;16]\\
t_4&=[3;1][4;2][7;5][9;6][11;8][12;10][15;13][16;14]\\
t_5&=[4,12;8][9,16;13]\\
t_6&=[2,6;5][4,9;7][10,14;13][12,16;15]\\
t_7&=[5,6;2][7,9;4][13,14;10][15,16;12]\\
t_8&=[8,11;3][10,12;4][13,15;7][14,16;9]\\
t_9&=[5;1][6;2][7;3][9;4][13;8][14;10][15;11][16;12]
\end{align*}

\subsection{Analysis of a permutator subsemigroup}

None of the above generators contain a cycle, so the existence of a nontrivial permutation group cannot be simply read off. The generated semigroup has 316665 elements. The decomposition of the semigroup shows has (several copies of) the following group components: cyclic group $C_2$ acting on 4, symmetric group $S_3$ acting on 3 and $C_2$ acting on 2 states.

We pick the set $\{3,5,8\}$ (there are many 3-element subsets that are mutually reachable from each other under the action of the semigroup, therefore they have isomorphic permutator groups). Computer calculation shows $|Perm(\{3,5,8\})|=542$. Consider the following  words of length 13 and 15, found by a breadth-first search, 
\begin{align*}
a =& t_1 t_5 t_3 t_8 t_5 t_1 t_4 t_8 t_5 t_7 t_8 t_5 t_6 \\
b=& t_1 t_4 t_8  t_5 t_3 t_8 t_5 t_1 t_4 t_8 t_5 t_7 t_8 t_5 t_6 
\end{align*}
realizing transformations
\begin{align*}
\overrightarrow{a}=&([1,2,10;\mathbf{3}],[4,7,9,11,12,15,16;\mathbf{5}],[6,13,14;\mathbf{8}])\\
\overrightarrow{b}=&[1,2,4;3]([10,11,12,13,14,15,16;\mathbf{5}],[6,7,9;\mathbf{8}]).
\end{align*}
As highlighted, these are clearly permutator words for the set $\{3,5,8\}$ and generate $S_3$.
It is straightforward to verify that these two words are straight.
Moreover, $a$ and $b$ can be checked to be minimal permutators (i.e.\ they cannot be properly factored into permutators of $\{3,5,8\}$).
However, the idempotent powers of these words $\overrightarrow{bb}$ and $\overrightarrow{aaa}$ are not equal, so the transformations do not lie in the same subgroup of the semigroup of the automaton.

We derive from these words, 
$x = b b a b b$  (a word with 73 letters), 
which reduces to straight word: $b b a$, (with only 43 letters), giving the transformation
$$\overrightarrow{x}=([1,2,10;5],[6,13,14;8])[4,7,9,11,12,15,16;3]$$
and 
$y = aaabaaa$, (another long word with 93 letters), which reduces to straight word  
$aaab$ (with 54 letters), giving the transformation
$$\overrightarrow{y}=([1,2,10;5],[6,13,14;8])[4,7,9,11,12,15,16;3].$$

These words ($a$,$b$, $aaab$, $bba$) are all straight permutator words, but obviously $aaab$ and $bba$ are not minimal permutators since
they are products of (straight) minimal permutators $a$ and $b$.

We have two copies of group  the symmetric group $S_3$ each faithfully acting on $\{3,5,8\}=\{M,R,C\}$:
one $S_3$ is generated by $a$ and $y$, and  another isomorphic copy of $S_3$ by $b$ and  $x$
with idempotents (the identity elements of these two groups): 
$\overrightarrow{a^3} = \overrightarrow{y^2} = [1,2,10;8][4,7,9,11,12,15,16;3][6,13,14;5] $
and 
$\overrightarrow{b^2} = \overrightarrow{x^3}  = [ 1,2,4;3][6,7,9;5][10,11,12,13,14,15,16;8]$, respectively.

Together the elements $\overrightarrow{a}$ and $\overrightarrow{b}$ generate a 12 element semigroup which is just the union of these two groups. This is to some extent counterintuitive, as one would expect one copy of the permutator group for one particular subset of the state set; furthermore as mentioned above the permutator semigroup $Perm(Y)$ has, not just these 12, but 542 elements, and this is but one of many instances of sets of states in this biological model acted on by the symmetric group  $S_3$.

\section{Conclusion}

Based on algorithmic efficiency considerations we studied straight words that encode loop-free paths in the Cayley graph of a transformation semigroup. We focused on straight words generating transformations that permute a given subset of the state set. 
We found that these minimal permutator straight words form a finite code, and also the minimal permutator words
form a code, although, as easy examples show, the latter is generally an infinite code.
The minimal permutator straight words  generate the corresponding subgroup of the transformation semigroup. These can be exploited in the calculations of hierarchical decompositions.  
These findings show that there are lot more ways within a semigroup to generate a subgroup than one might think, but for finding the subgroup it is enough to consider a subset of them.

\bibliographystyle{splncs03}
\bibliography{straightwords}

\end{document}